\documentclass[final]{siamart190516}
\usepackage[utf8]{inputenc}
\usepackage[english]{babel}
\usepackage{amsmath, mathtools, amsfonts, amssymb}
\usepackage{graphicx}
\usepackage{lmodern}
\usepackage{enumitem}
\usepackage{color}
\usepackage{cleveref}

\usepackage{tikz}
\usetikzlibrary{calc}
\usetikzlibrary{shapes,arrows}

\usepackage{algorithm}
\usepackage{algpseudocode}
\usepackage{algorithmicx}
\usepackage{varwidth}
\usepackage{siunitx}

\DeclareMathOperator{\Div}{div}
\newcommand{\R}{{\mathbb{R}}}

\renewcommand{\L}{{\mathcal{L}}}

\DeclareMathOperator{\id}{id}
\newcommand{\Gobs}{{\Gamma_{\text{d}}}}
\newcommand{\hatGobs}{{\hat{\Gamma}_{\text{d}}}}
\newcommand{\Gout}{{\Gamma_{\text{out}}}}
\newcommand{\Gin}{{\Gamma_{\text{in}}}}
\newcommand{\Gns}{{\Gamma_{\text{ns}}}}

\newcommand{\Oobs}{{\Omega_{\text{d}}}}
\newcommand{\hatOobs}{{\hat{\Omega}_{\text{d}}}}

\newcommand{\holdall}{{U}}

\newcommand{\bc}{{\text{bc}}}
\newcommand{\vol}{{\text{vol}}}

\newcommand{\lambdav}{\lambda}
\newcommand{\lambdabc}{\mu}

\newcommand{\mtestw}{{\psi_w}}
\newcommand{\mtestp}{{\psi_p}}
\newcommand{\mtestv}{{\psi_v}}
\newcommand{\mtestb}{{\psi_b}}
\newcommand{\control}{{c}}

\newcommand{\Next}{{n_{\text{ext}}}}
\newcommand{\ainit}{{\alpha_{\text{init}}}}
\newcommand{\atarget}{{\alpha_{\text{target}}}}
\newcommand{\adec}{{\alpha_{\text{dec}}}}

\newcommand{\deltawsec}{{\bar{h}_w}}
\newcommand{\deltaw}{{h_w}}
\newcommand{\deltav}{{h_v}}
\newcommand{\deltap}{{h_p}}
\newcommand{\deltab}{{h_b}}
\newcommand{\deltatestv}{{h_\mtestv}}
\newcommand{\deltatestp}{{h_\mtestp}}
\newcommand{\deltatestw}{{h_\mtestw}}
\newcommand{\deltatestb}{{h_\mtestb}}
\newcommand{\deltacontrol}{{h_\control}}
\newcommand{\deltalambdav}{{h_{\lambdav}}}
\newcommand{\deltalambdabc}{{h_\lambdabc}}

\newsiamremark{remark}{Remark}

\DeclareMathOperator{\Tr}{Tr}

\definecolor{blue}{RGB}{0,101,189} 
\definecolor{blues1}{RGB}{0,82,147}
\definecolor{blues2}{RGB}{0,51,89}
\definecolor{bluea1}{RGB}{152,198,234}
\definecolor{bluea2}{RGB}{100,160,200}
\definecolor{elfea}{RGB}{218,215,213} 
\definecolor{oraa}{RGB}{227,114,34} 
\definecolor{grea}{RGB}{162,173,0}

\date{\today}
\title{A continuous perspective on modeling of shape optimal design problems}
\author{J. Haubner\thanks{Department of Mathematics, Technical
		University of Munich, Boltzmannstr.~3, 85748 Garching b. M\"unchen,
		Germany \mbox{(haubnerj@ma.tum.de)}} \and M. Siebenborn\thanks{University of Hamburg, Department of Mathematics, Bundesstr.~55, 20146 Hamburg, Germany \mbox{(martin.siebenborn@uni-hamburg.de)}}
	    \and M. Ulbrich \thanks{Department of Mathematics, Technical
	    	University of Munich, Boltzmannstr.~3, 85748 Garching b. M\"unchen,
	    	Germany \mbox{(mulbrich@ma.tum.de)}}
}

\begin{document}
\maketitle



\begin{abstract}
  
In this article we consider shape optimization problems as optimal control problems 
via the method of mappings. Instead of optimizing over a set of admissible shapes 
a reference domain is introduced and it is optimized over a set of admissible 
transformations. The focus is on the choice of the set of transformations, 
which we motivate from a function space perspective.  
In order to guarantee local injectivity of the admissible transformations 
we enrich the optimization problem by a nonlinear constraint.
The approach requires no parameter tuning for the extension equation 
and can naturally be combined with geometric constraints 
on volume and barycenter of the shape. Numerical results for 
drag minimization of Stokes flow are presented.
\end{abstract}

\section{Introduction}

Shape optimal design is a vivid research field with a wide range of applications from fluid-dynamics \cite{schmidt2013three, BaLiUl, garcke2016stable}, acoustics \cite{udawalpola2008optimization}, electrostatics \cite{langer2015shape}, image restoration and segmentation \cite{hintermuller2004second}, interface identification in  transmission  processes \cite{schulz2015structured, harbrecht2013numerical, naegel2015scalable} and nano-optics \cite{hiptmair2018large} to composite material identification \cite{siebenborn2017algorithmic, naegel2015scalable}.

In shape optimization, a shape functional $\tilde j: \mathcal O_\text{ad} \to \mathbb R$ is optimized over a set of admissible shapes $\mathcal O_\text{ad}$, i.e., 
\begin{align}
\min_{\Omega \in \mathcal O_\text{ad}} \tilde j(\Omega).
\label{shapeopt}
\end{align}
There are various ways to tackle this problem. 
In this work, 
we focus on the method of mappings \cite{MuSi, BaLiUl, KeUl,  FLUU16}. Here, the optimization problem \eqref{shapeopt} is
reformulated as an optimization problem over a set of admissible transformations $\mathcal T_\text{ad}$
defined on a nominal domain $ \Omega$:
\begin{align}
\min_{\tau \in \mathcal T_\text{ad}}  j(\tau),
\end{align}
where $ j (\tau) \coloneqq \tilde j(\tau (\Omega))$. This approach is closely
related to techniques that use shape gradients and the Hadamard-Zol\'esio structure
theorem. 

Mesh degeneration is one of the bottlenecks in performing transformation-based 
shape optimization techniques, see e.g. \cite{EHLG18}. 
On the one hand, by the modeling of the optimization problem 
it has to be ensured that the boundary of the transformed domain 
is not self-intersecting. This can, e.g., be realized using bounds on the deformation or geometrical constraints, 
such as volume and barycenter constraints. 
On the other hand, mesh degeneration also appears for large deformations of the 
surface even if the boundary of the domain is not self-intersecting. 
Therefore, finding transformations that preserve the mesh quality is an active field of research. In \cite{IgStWe} it is proposed to work with an extension equation that preserves the mesh quality. This method, however, is limited to 2d cases. Another
approach is remeshing, see e.g. \cite{Wilke2005,EF16,B17}. 
The quality of the mesh can be improved by using a function $\psi(w)$ such that $\tau(\Omega) = (\id + \psi(w))(\Omega) = \Omega$, where $\psi(w)$ is either defined via the solution of a partial differential equation or via a solution of an optimization problem. Both methods allow for node relocations without changing $\Omega$ 
and hence are so called r-refinement strategies. 
Other approaches project the shape gradient to mimic the continuous behaviour 
motivated through the Hadamard-Zol\'esio structure theorem \cite{EHLG18}
or
work with extension equations that require parameter tuning 
in order to avoid mesh degeneration \cite{schulz2016computational, schulz2016efficient, langer2015shape}. However, finding adequate parameters for a given extension equation tends to be a time consuming effort. Moreover, the empirically determined parameters are typically tailored for one specific mesh and problem setting.

The starting point for our considerations is the fact that the second type of 
mesh degeneration is a phenomenon that only appears in the discretized setting. 
Thus we consider the problem from a continuous perspective and require
sufficient high regularity of the boundary deformations analogous to 
\cite{Kunisch2001, Slawig2004,KiVe13, BaLiUl} 
where parametrizations of the design boundary with sufficiently 
high regularity are used. 
Instead of preserving mesh quality, 
our approach ensures that all admissible controls yield transformations 
that map the reference domain $\Omega$ to a Lipschitz domain. 
Since the optimization problem is formulated in the continuous setting, 
this approach also allows for refinement and remeshing
techniques, wheareas from a discretized point of view, remeshing also requires a reinitialization of the optimization algorithm.
However, an accurate modeling remains challenging since, on the one hand, the most general setting, i.e., working with transformations in
$W^{1, \infty}( \Omega)^d$, is difficult since it is a non-reflexive Banach space. On the other 
hand, working with smoother spaces often requires $H^2$-conforming finite element methods as used in \cite{KiVe13}. 

In this work, we focus on the modeling of the shape optimization problem 
respecting the continuous requirements on the transformations. 
Motivated by the theoretical considerations in \cref{sec::generalformulation}, 
we consider Banach spaces $\tilde X, X, Y$ such that $ X \hookrightarrow \tilde X$
and $ Y \hookrightarrow \mathcal C^1(\overline \Omega)^d$ and a mapping 
$S$ that is continuous as a mapping $S: X \to Y$ and $S: \tilde X \to \mathcal C^1
(\overline \Omega)^d$.
In addition, we enrich the optimization problem with additional 
constraints and investigate 
\begin{align}
\begin{split}
\min_{\control \in X} & ~j (\id + w) + \frac\alpha{2} \| \control \|_{X}^2 
\\
\text{s.t. } &g(w) = 0,\\
& w = S(\control), \\
&\| \control \|_{ \tilde X} \leq \eta_2, \\
&\det(\nabla (\id + w)) \geq \eta_1 \quad \text{in }  \Omega, \\
\end{split}
\label{optprob2}
\end{align}
for $\eta_1 \in (0,1)$, $\eta_2 \geq 0$ where $g$ represents geometric constraints. 
We choose $S$ such that the requirements are fulfilled in two and three 
dimensions and work on Hilbert spaces. 
Therefore, we require $Y \hookrightarrow H^{\frac52 + \epsilon}(\Omega)$ 
with $\epsilon > 0$. To circumvent the use of 
$H^2$-conforming finite elements the regularity is lifted step-wise. 
In this paper, we focus on an approach that starts with a 
design parameter $\control \in L^2(\Gobs)$ that is mapped to a function 
$b \in H^2(\Gobs)$ by solving a Laplace-Beltrami equation. 
Imposing $b$ as Neumann boundary 
condition for an elliptic extension equation we obtain a deformation field $w$. 
However, there are various other possibilities. 
Alternatively, one could also start with $\control \in H^1(\Gobs)$ and impose $b$ as Dirichlet boundary condition for the elliptic extension equation. 
Compared to 
previous approaches, the only difference is the additional Laplace-Beltrami equation, which ensures sufficiently high regularity of the deformation field, and the additional
nonlinear constraint. This allows us to 
integrate this new approach without much effort into existing methods.

To test the formulation numerically, 
we focus on shape optimization for the steady state Stokes flow, see
e.g. \cite{mohammadi2010applied}. \cref{geo::conf} illustrates the geometrical configuration that we use as reference domain. We consider a rectangular domain with an obstacle in the center, which has a smooth boundary $\Gobs$, i.e.\ the design boundary.
With $\Oobs$ we denote the domain encircled by $\Gobs$.
On the left boundary of the domain $\Gin$ Dirichlet boundary conditions and on the right boundary $\Gout$ do-nothing boundary conditions are imposed. On the rest of the boundary no-slip boundary conditions are imposed. We optimize the shape of the obstacle via the method of mappings 
such that the drag is minimized.
\begin{figure}[ht!]
	\centering
	\begin{tikzpicture}
	\draw (0.0, 0.0) -- (5.0,0.0);
	\draw (0.0, 0.0) -- (0.0,3.4);
	\draw (5.0, 0.0) -- (5.0,3.4);
	\draw (0.0, 3.4) -- (5.0,3.4);
	\draw (2.5, 1.7) circle (20pt); 
	\draw (2.95, 1.7) node {$\Gobs$};
	\draw (-0.3, 1.7) node {$\Gin$};
	\draw (5.35, 1.7) node {$\Gout$};
	\draw (2.5, 3.65) node {$\Gns$};
	\draw (4.5, 2.9) node {$\Omega$};
	\end{tikzpicture}
	\caption{2d sketch of the geometrical configuration for a shape optimization problem that is governed by Stokes flow.}
	\label{geo::conf}
\end{figure}
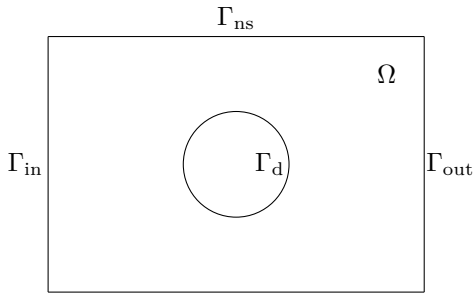

\Cref{sec::generalformulation} is devoted to the general formulation of the shape optimization problem. \Cref{sec::s1} motivates the 
validity of this approach by theoretical considerations for a special choice for the control-to-deformation mapping. 
\Cref{sec::Stokes} presents the application of the abstract framework to the Stokes flow example. 
Also other strategies for the control-to-deformation mapping 
are presented and only tested numerically.
An algorithmic realization for solving this optimization problem is given in \cref{sec::optsys}. Numerical results in \cref{sec::numerics} show the performance of the different strategies. 

\section{Shape Optimization Problem on Function Space}
\label{sec::generalformulation}
We consider the following optimization problem 
\begin{align}
\begin{split}
\min_{\control \in D_\text{ad}} &j (\tau) 
\\
\text{s.t. } &\tau = \id + w,\\ &g(w) = 0,\\
& w = S(\control),
\end{split}
\label{optprob}
\end{align}
where $g(w)$ represents geometric constraints. 
The design parameter is denoted by $\control$
and the corresponding transformation is defined via $\tau \coloneqq \id + w$.
Moreover, $D_\text{ad} \subset L^2(\Gamma)$ 
and $S$ are chosen such that the following assumptions hold true.
\begin{enumerate}[label=\textbf{A\arabic*}]
	\item For all admissible controls $\control \in D_\text{ad}$ there exists an open neighborhood $U$ of $\Omega$ and a $\mathcal C^1$-diffeomorphism $F: U \to U$ such that $F\vert_{\Omega} = \id + S(\control)$ a.e.. \label{assumption2}
	\item Let $\control_1, \control_2 \in D_\text{ad}$. Then $(\id + S(\control_1))(\Omega) = (\id + S(\control_2))(\Omega)$ if and only if $\control_1 = \control_2$ a.e.. 
	\label{assumption1}
\end{enumerate}
The second assumption \ref{assumption1} guarantees that there is a one-to-one correspondence between shapes and controls. 
The first assumption \ref{assumption2} ensures that $\id + w$ is the restriction of a $\mathcal C^1$-diffeomorphism that maps an open neighborhood of $\Omega$ to itself and implies the following lemma. 
\begin{lemma}
	Let $\Omega$ be a smooth domain, 
	and assumption \ref{assumption2} be fulfilled.  
	Then $(\id + S(\control))(\Omega)$ is a Lipschitz-domain for all admissible $\control \in 
    D_\text{ad}$.
\end{lemma}
\begin{proof}
	Follows directly from \cite[Thm. 4.1]{Hofmann2007}.
\end{proof}

\subsection{On the choice of $D_\text{ad}$ and $S$}
\label{sec::motivation}
Inspired by \cite[Lem. 4]{HUU19}, we present sufficient conditions for assumption \ref{assumption2} to be fulfilled. The following extension property will be a helpful
tool.
\begin{lemma}
	Let $d \in \lbrace2,3\rbrace$, $\Omega$ be a bounded Lipschitz domain,
	$\eta_1 \in (0,1)$.
	Furthermore, let $X, \tilde X, Y$ be Banach spaces such that 
	$Y  \hookrightarrow \mathcal C^1(\overline \Omega)^d$, $X \hookrightarrow \tilde X \hookrightarrow L^2(\Gamma)$ and 
	$S: X \to Y$, $S: \tilde X \to \mathcal C^1(\overline \Omega)^d$ 
	be continuous. Then, there exists $\eta_2  > 0$ such that for 
	\[
	D_\text{ad} \coloneqq \lbrace \control \in X ~:~ \det(\nabla (\id + S(\control))) > \eta_1, ~
	\| \control \|_{\tilde X} \leq \eta_2 \rbrace,
	\]
	assumption \ref{assumption2} holds true.
	\label{lem::fulfillsassumptions} 
\end{lemma}
\begin{proof}
	Let $\control \in D_{ad}$ be feasible and $\tau_{\control}: \Omega \to \tau_{\control}(\Omega)$, $\tau_{\control} \coloneqq \id + S(\control)$. 
	We know that $S(\control) \in Y$ which embeds 
	into $\mathcal C^1(\overline{\Omega})^d$. 
	Moreover, there exists a constant $C_S>0$ such that \begin{align}\| S(\control)\|_{\mathcal C^1(\overline{\Omega})^d} \leq C_S \| \control \|_{\tilde X} \label{eq::cont} \end{align} for all $\control \in D_\text{ad}$. 
	
	By the constraint $\det( \nabla \tau_{\control}) \geq \eta_1$ we know that $\tau_\control$ is a local diffeomorphism. For $\tau_{\control}$ to be a global 
	diffeomorphism bijectivity of $\tau_{\control}$ has to be ensured, see \cite[Sec. 2, p. 36]{Lee}. Since surjectivity holds by definition of $\tau_\control$, it remains to show injectivity. This can be achieved
	by choosing $\eta_2$ sufficiently small such that $\| S(\control)\|_{W^{1, \infty}(\Omega)^d} < 1$. In fact, assuming that there exist $x_1, x_2 \in \Omega$ such that $\tau_{\control}(x_1) - \tau_{\control}(x_2) = 0$ implies 
	\begin{align}
	\| x_1 - x_2 \| = \| S(\control)(x_1) - S(\control)(x_2) \| \leq \|S(\control)\|_{W^{1, \infty}(\Omega)^d} \|x_1 - x_2 \|, 
	\label{eq::inj}
	\end{align}
	and hence $x_1 = x_2$ which yields injectivity. By using the inverse function theorem it can be shown that $\tau^{-1}_{\control}$ is $\mathcal C^1$ for all $\eta_2>0$ 
	sufficiently small, see also \cite[Lem. 4]{HUU19}. 
	
	In order to fulfill assumption \ref{assumption2} we have to be able to extend $\tau_{\control}$ to a $\mathcal C^1$-diffeomorphism $F: U \to U$ where $U$ is an open neighborhood of $\overline \Omega$.
	
    By \cite[Thm. 2.74, (2.145)]{Brudnyi} for $k \in \mathbb N_0$, 
    there exists an extension operator 
    $\mathrm{Ext}: \mathcal C(\overline \Omega) \to \mathcal C(\mathbb R^d)$ such that 
    $\mathrm{Ext}(\mathcal C^\ell(\overline \Omega)) \subset \mathcal C^\ell (\mathbb R^d)$ 
    for all $\ell \in \lbrace 0, \ldots, k\rbrace$ and such that there exists $\tilde C >0$ with
    \[
    \max_{|\alpha| = \ell} \sup_{x \in \mathbb R^n} | D^\alpha \mathrm{Ext} (f) (x) | 
    \leq \tilde C \| f \|_{\mathcal C^\ell(\overline \Omega)} \quad \forall f \in \mathcal C^\ell(\overline \Omega)
    \]
    for all $\ell \in \lbrace 0, \ldots, k\rbrace$. Hence there exists an extension $\tilde w$ and a constant $C_\text{ext} >0 $ such that
    \begin{align}
    \| \tilde w \|_{\mathcal C^1(\mathbb R^d)^d} \leq C_\text{ext} \| S(c) \|_{\mathcal C^1(\overline \Omega)^d}
    \label{pro:ext}
    \end{align}
    and $\tilde w \vert_{\Omega} = S(c)$.
	We choose $\alpha >0$ and set $U := B_\alpha(\Omega)$. Let 
	$\varphi := 1_{B_{\frac{\alpha}2}(\Omega)} \ast \psi $ be the convolution of
	the indicator function $1_{B_{\frac{\alpha}2}(\Omega)}$ of $B_{\frac{\alpha}2}(\Omega)$ and a mollifier $\psi \in \mathcal C^\infty(\mathbb R^d)$ such that
	$\int_{\mathbb R^d} \psi dx = 1$ and $\mathrm{supp}(\psi) \subset B_{\frac{\alpha}4}(0)$. Hence, $\varphi \in \mathcal C^\infty(\mathbb R^d)$ and there exists $C_{\alpha} > 0$ such that
	\begin{align}
	\| \varphi \|_{\mathcal C^1(\mathbb R^d)} \leq C_{\alpha}.
	\label{pro:conv}
	\end{align}
	Define
	$F(x):= \mathrm{id} + \tilde w \varphi$, which is an element of $\mathcal C^1(\overline\Omega)^d$.
	By \eqref{pro:conv}, \eqref{pro:ext}, \eqref{eq::cont} and the definition of $D_{ad}$
	there exists $C>0$ such that
	\begin{align}
	\| \tilde w \varphi \|_{\mathcal C^1(\mathbb R^d)^d} \leq C C_\text{ext}C_{\alpha} C_S \eta_2.
	\label{pro:est}
	\end{align}
	Possibly reducing $\eta_2$ such that $\eta_2 < (C C_\text{ext} C_{\alpha} C_S)^{-1}$ implies injectivity of $F: \mathbb R^d \to \mathbb R^d$ analogous to \eqref{eq::inj}. By definition, $\varphi = 0$ on $ \mathbb R^d \setminus U$ and hence $F(\mathbb R^d \setminus U)
	= \mathbb R^d \setminus U$.  Due to injectivity of $F: \mathbb R^d 
	\to \mathbb R^d$ there is no $x \in U$ such that $F(x) \in \mathbb R^d
	\setminus U$. Thus, $F(U) \subset U$ and
	$F: U \to U$ is injective. 
	Furthermore, $F$ is a local diffeomorphism after 
	possibly again 
	reducing $\eta_2$ since there exists a constant $\tilde C >0$ such that
	\begin{align}
	\begin{split}
	\mathrm{det}(\nabla F(x))  &\geq 1 - \|  \mathrm{det}({\nabla F(x)}) -  \mathrm{det}(\nabla \id( x)) \|_{C(\mathbb R^d)^d} \\ &
	\geq 1 - \tilde C \| \tilde w \varphi \|_{\mathcal C^1(\mathbb R^d)^d} \geq  1 - \tilde C C C_\text{ext} C_\alpha C_S \eta_2 
	\end{split}
	\end{align}
	for all $x \in \mathbb R^d$ where we used \eqref{pro:est} and that the determinant is a polynomial 
	of degree $d$ in the entries of the matrix where $d$ denotes the dimension.
	
	We now show surjectivity.
	Since $\overline U$ is compact and $F$ is continuous, $ F(\overline U)$ is compact. Assume that 
	$F: U \to U$ is not surjective, then there exists $\tilde x \in U$ s.t. $\tilde x \notin F(U)$. Since $F(\partial U )= \partial U$ 
	($F$ acts like the identity on $\partial U$) and $U$ is open, $\tilde x \notin 
	F(\overline U)$. Since $F(\overline U)$ is compact and $F$ is continuous,
	there exists
	$\bar x \in \mathrm{argmin}_{x \in F(\overline U)} 
	\frac12 \| x - \tilde x \|_2^2$. By the choice of $\bar x$, $ \bar x + t (\tilde x - \bar x)
	\notin F(\overline U)$ for all $t \in (0,1]$. Furthermore, 
	$ \bar x + t (\tilde x - \bar x) \in U$ for all $t  \in (0,1]$, since otherwise
	there would exist $ \tilde t \in (0,1)$ such that $\bar x + \tilde t (\tilde x - \bar x)
	 \in \partial U = F(\partial U) \subset F(\overline U)$.
	This implies $  \bar x + t (\tilde x - \bar x) \notin \mathbb R^d\setminus U = F(\mathbb R^d \setminus U)$ for all $t \in (0,1]$. Therefore, $ \bar x + t (\tilde x - \bar x)
	\notin F(\mathbb R^d)$ for all $t \in (0,1]$ and
	$B_\epsilon(\bar x) \not \subset F(\mathbb R^d)$ for all $\epsilon > 0$. 
	This contradicts $F: \mathbb R^d \to 
	\mathbb R^d$ being a local diffeomorphism since, for $\bar y \in 
	\mathbb R^d$ such that $F(\bar y) = \bar x$ (which exists since $\bar x \in F(\overline U)$), there exists an open neighborhood of $\bar y$ that is diffeomorphically mapped to an open neighborhood of
	$\bar x$.
	
	Thus, 
	we have shown that $F$ is a bijective local 
	diffeomorphism. Hence, $F$ is a global diffeomorphism and 
	$\mathcal C^1$-regularity of the inverse is again obtained as in \cite[Lem. 4]{HUU19}
	by possibly again reducing $\eta_2$. 
	Therefore, $F: U \to U$ is a 
	$\mathcal C^1$-diffeomorphism.
\end{proof}

\begin{remark}
	Alternatively, if one provides a mesh for the hold all domain $\holdall$, $w$ and the constraint $\det(\nabla (\id + w))$ can be defined on $\holdall$.
\end{remark}

\Cref{lem::fulfillsassumptions} motivates to consider optimization problems of the form
\eqref{optprob2}.
\subsection{Displacement along normal directions}
\label{sec::s1}
In order to avoid technicalities we consider a smooth domain $\Omega$.
Furthermore, we assume that $\Gamma\setminus \Gobs \neq \emptyset$.
In this section we consider $S(\control) \coloneqq S_\Omega(S_{\Gobs}(\control)) \Next$, where 
\begin{itemize}
	\item $\Next$ is a smooth extension of the outer unit normal vectors to $\Omega$,
	\item $S_{\Gobs}$ is the solution operator of the Laplace-Beltrami equation on $\Gobs$
	\[
	- \Delta_{\Gobs} b + b = f \quad \text{on } \Gobs,
	\]
	\item $S_\Omega$ is the solution operator of the elliptic equation 
	\begin{align*}
	- \Delta z = 0 \quad & \text{in }  \Omega, \\
	z = 0 \quad & \text{on }  \Gamma \setminus  \Gobs, \\
	\nabla z \cdot n = b  \quad & \text{on }  \Gobs.
	\end{align*}
\end{itemize} 

In correspondence with numerical examples that we consider in \cref{sec::numerics}, we assume 
$ \Gobs$ to be a compact manifold without boundary. 
Using \cref{lem::fulfillsassumptions} we prove that the assumptions \ref{assumption2} and \ref{assumption1} are 
fulfilled if $\Next$ and the Banach space $X$ are chosen in an appropriate way, see \cref{lem::tildeS}. To this end, we recall well-known results for the elliptic solution operators.
\begin{lemma}[Elliptic equation on compact manifolds without boundary]
	Let $s \geq -1$, $ \Gobs$ be a smooth and compact Riemannian manifold without boundary 
	and consider the system
	\begin{align}
	- \Delta_{\Gobs} b + b = f 
	\label{elliptic_on_bdry}
	\end{align}
	on $ \Gobs$, where $ \Delta_{\Gobs}$ denotes the Laplace-Beltrami operator on 
	$ \Gobs$. Then, for any $f \in H^s( \Gobs)$ there exists a unique solution 
	$b \in H^{s+2}( \Gobs)$ and the corresponding solution operator $S_{\Gobs}: H^s( \Gobs) \to H^{s+2}( \Gobs)$ is continuous.
	\label{lem::laplacebeltrami}
	
\end{lemma}
\begin{proof}
	See \cite[pp.362-363]{TaylorPDEI}.
\end{proof} 
Since $ \Gobs$ is closed and has positive distance from $ \Gamma \setminus \Gobs$, classical results for the Dirichlet and Neumann boundary value problem also hold for the mixed boundary value problem in our setting whereas it gets more involved when the positive distance assumption is not fulfilled, see, e.g., 
\cite{Lieberman86}.
\begin{lemma} Let $ \Omega$ be a smooth domain and $ \Gobs \subset  \Gamma$ be a closed subset of the boundary such that $ \Gamma \setminus  \Gobs \neq 
	\emptyset$. Assume that $ \Gobs$ and $ \Gamma\setminus \Gobs$ have
	positive distance. Let $s \geq 2$. 
	Consider the following system 
	\begin{align}
	\begin{split}
	- \Delta z = 0 \quad & \text{in }  \Omega, \\
	z = 0 \quad & \text{on }  \Gamma \setminus  \Gobs, \\
	\nabla z \cdot n = b  \quad & \text{on }  \Gobs.
	\end{split}
	\label{eq::ell}
	\end{align}
	Then, for every $b \in H^{s-\frac32}( \Gobs) $ there exists a unique solution $z \in H^s( \Omega) $ and the corresponding solution 
	operator $S_\Omega: H^{s-\frac32}( \Gobs) \to H^s( \Omega)$ is continuous. 
	\label{lem::help2}
\end{lemma}
\begin{proof}
	see \cite[p.188, Rem. 7.2]{LM_v1}.
\end{proof}
These two lemmas imply that assumptions \ref{assumption2} and \ref{assumption1} are fulfilled for the choice $\tilde S = S_\Omega\circ S_{\Gobs}$ and 
$X = H^1(\Gobs)$ as the following lemma shows.
\begin{lemma}
	Let $\Omega$ be a bounded smooth $\mathcal C^\infty$-domain and $X = \tilde X = L^2(\Gobs)$.
	Let $\tilde S (\control) \coloneqq S_\Omega(S_{\Gobs} (\control))$ for all $\control 
	\in X$. 
	Then there exists $\eta_2 >0$ such that assumptions \ref{assumption2} and \ref{assumption1} are fulfilled for 
	$S(\cdot) = \tilde S(\cdot) \Next$ for $D_{ad}$ 
	chosen as in Lemma \ref{lem::fulfillsassumptions}.
	\label{lem::tildeS}
\end{lemma}
\begin{proof}
	By \cref{lem::laplacebeltrami} and \cref{lem::help2}, 
	$S_\Omega(S_\Gamma(X)) \subset H^{\frac72}(\Omega)$,
	which embeds into $\mathcal C^1(\overline{\Omega})$. Thus, 
	$\tilde S$ fulfills the requirements of
	\cref{lem::fulfillsassumptions} and assumption \ref{assumption2}
	holds.
	Let $\control_1, \control_2 \in X$ and $S(\control_1)(\Omega) = S(\control_2)(\Omega)$. 
	Then, $\tilde S(\control_1)\vert_{\Gobs} 
	= \tilde S(\control_2)\vert_{\Gobs}$. Linearity and 
	well-definedness of the Neumann-to-Dirichlet map for the elliptic equations 
	\eqref{eq::ell}, see, e.g., \cite{K04}, implies $S_{\Gobs}(\control_1) = 
	S_{\Gobs}(\control_2)$. Thus, due to linearity of $S_{\Gobs}$, $\control_1 = \control_2$ a.e.
	and assumption \ref{assumption1} is fulfilled.
\end{proof}

\section{Example: Stokes flow}
\label{sec::Stokes}
We now apply \eqref{optprob2} to minimize
the drag of an obstacle in steady-state Stokes flow, see \cref{geo::conf}.
The optimization problem is given by
\begin{align}
\begin{split}
\min_{\control \in L^2( \Gobs)} &\frac12 \int_{\tau(\Omega)}( \nabla v : \nabla v) dx + \frac\alpha{2} \| \control \|_{L^2( \Gobs)}^2 \\
\text{s.t. } & {\begin{cases}
\Delta v + \nabla p = 0 \quad & \text{in } \tau(\Omega), \\
\Div(v) = 0 & \text{in }\tau(\Omega),\\
v = 0 &\text{on } \tau(\Gobs) \cup \Gns, \\
v = g_{in} &\text{on } \Gin, \\
{(\nabla v - p I) n= 0} & \text{on } \Gout,
\end{cases}}\\
&\tau = \id + w, \\
&w = S(\control),\\
&g(w) = 0,\\
&\det(\nabla \tau) \geq \eta_1 \quad \text{in }  \Omega.
\end{split}
\label{optprobstokes}
\end{align}
Here, $v$ denotes the fluid velocity, $p$ the fluid pressure and $g_{in}$ 
non-homogeneous Dirichlet boundary conditions on $\Gin$ and $S$ is chosen such
that the trace $S(d)\vert_{\Gns \cup \Gin \cup \Gout} = 0$ for all admissible $d \in L^2(\Gobs)$.
In order to exclude trivial solutions we add geometric constraints $g(w)=0$ to the optimization problem \cref{optprobstokes}, which are further discussed in \cref{sec::geocon}. The additional norm constraint on $\control$ is not 
crucial for the numerical implementation of this problem and is therefore neglected.

\subsection{Algorithmic realization}
We want to use state-of-the-art finite element toolboxes to solve the optimization 
problem. This can, e.g., be realized by penalizing the inequality constraints. Hence, we
obtain the equality constrained optimization problem:
\begin{align}
\begin{split}
\min_{\control \in L^2( \Gobs)} &\frac12 \int_{\tau(\Omega)}( \nabla v : \nabla v) dx + \frac\alpha{2} \| \control \|_{L^2( \Gobs)}^2 + \frac{\gamma_1}2 \| (\eta_1 - \det(\nabla \tau))_+\|^2_{L^2(\Omega)} \\
\text{s.t. } & {\begin{cases}
	\Delta v + \nabla p = 0 \quad & \text{in } \tau(\Omega), \\
	\Div(v) = 0 & \text{in }\tau(\Omega),\\
	v = 0 &\text{on } \tau(\Gobs) \cup \Gns, \\
	v = g_{in} &\text{on } \Gin, \\
	{(\nabla v - p I) n= 0} & \text{on } \Gout,
	\end{cases}}\\
&\tau = \id + w, \\
&w = S(\control),\\
&g(w) = 0,
\end{split}
\label{optprobstokespen}
\end{align}
where $\gamma_1 >0$ denotes a penalization parameter and 
$(\cdot)_+ \coloneqq \max(0, \cdot)$. In order to simplify the notation, we will use the notation $J_\tau \coloneqq \det(D\tau)$ in the sequel.
The first order necessary optimality conditions of \cref{optprobstokespen} yield a system of nonlinear, coupled PDEs, see \cref{sec::optsys}.

\begin{algorithm}[h!]
	\caption{Optimization strategy}
	\label{alg::opt_str}
	\begin{algorithmic}[1]
		\Require $0 < \atarget \leq \ainit$, $0 < \adec < 1$, $0 \leq \gamma_1$, 
		$0 < \eta_1$
		\State $k \gets 0$, $\alpha_k \gets \ainit$, $\control_k \gets 0$
		\While {$\alpha_k \geq \atarget$}
		\State Solve \eqref{optprobstokespen} iteratively with initial point $\control_k$ and solution $\control$		
		\State $\alpha_{k+1} \gets \adec \alpha_k$, $\control_{k+1} \gets \control$
		\State $k\gets k+1$
		\EndWhile
	\end{algorithmic}
\end{algorithm}

In principle, one solution of a nonlinear system of PDEs leads to the desired optimal solution for a given $\atarget$.
From a computational point of view, yet, the solvability of this system with semismooth Newton methods depends on the initialization. Therefore, we solve 
\eqref{optprobstokespen} for a sequence of decreasing regularization parameters, 
see \cref{alg::opt_str}. The following sections are devoted to explicitly derive the 
optimality system of \eqref{optprobstokespen} in a weak form, see \cref{sec::optsys}. Therefore, the 
geometrical constraints (\cref{sec::geocon}) are discussed and 
the different strategies for the control-to-transformation
mapping $S$ are investigated in more detail.

\subsection{Geometrical constraints}
\label{sec::geocon}

For shape optimization in the context of fluid dynamics it is necessary to fix the test specimen in space to avoid design improvements by moving it to the walls of the flow tunnel or shrinking it to a point.
In our situation this is to fix volume and barycenter of the obstacle body $\Oobs$.
In the following we use the symbol $\hat{\cdot}$ to refer to the deformed geometrical entity in terms of the mapping $\tau$.
If, for instance, $\Omega$ denotes the reference domain, then $\hat{\Omega} \coloneqq \tau (\Omega)$.

Let $\holdall$ be the hold all domain and the obstacle $\hatOobs = \holdall \setminus \hat{\Omega}$.
Further let 
\begin{equation}
\label{eq::vol_bc}
\vol(\hatOobs) = \int_\hatOobs 1 \, d\hat x, \quad \bc(\hatOobs) = \frac{1}{\vol(\hatOobs)}\int_\hatOobs \hat x\, d \hat x
\end{equation}
denote volume and barycenter of the obstacle.


In the numerical implementation we work with the corresponding boundary integral formulations instead. Let $\hat{n} : \hatGobs \rightarrow \R^d$ be the unit normal on $\hatGobs$ and $f \in L^1(\hatGobs)$. According to \cite[Prop. 2.47, Prop. 2.48]{sokolowski2012introduction}, we have
\begin{equation}
\label{eq::boundary_integral_trafo}
\int_{\hat{\Gamma}} \hat{f} \, ds(\hat x) = \int_\Gamma f \Vert J_\tau(D\tau)^{-\top} n \Vert_2 \, ds(x).
\end{equation}
Furthermore, the normal vector on the deformed boundary $\hatGobs$ is given in terms of the normal vector $n$ on the boundary of the reference domain $\Gobs$ as
\begin{equation}
\label{eq::normal_trafo}
\hat{n} \circ \tau = \frac{1}{\Vert (D\tau)^{-\top} n \Vert_2} (D\tau)^{-\top} n.
\end{equation}
Applying \eqref{eq::boundary_integral_trafo} and \eqref{eq::normal_trafo} to \eqref{eq::vol_bc} we obtain
\begin{equation}
\begin{aligned}
\vol(\hat{\Omega}) &= \int_{\hat{\Omega}} 1 \, d\hat x = \frac{1}{d} \int_{\hatGobs} \hat x^\top \hat{n} \, d \hat s(\hat x) \\ &= \frac{1}{d} \int_{\Gobs} (x+w)^\top (\hat{n} \circ \tau) \Vert J_\tau (D\tau)^{-\top} n \Vert_2 \, ds(x)\\
&= \frac{1}{d} \int_{\Gobs} (x+w)^\top (D\tau)^{-\top} n \vert J_\tau \vert \, ds(x).
\end{aligned}
\end{equation}
for the volume and
\begin{equation}
\label{eq::bc_surface}
\begin{aligned}
&(\bc(\hatOobs))_i = \frac{1}{\vol(\hatOobs)} \int_{\hatOobs} \hat x_i \, d\hat x = \frac{1}{\vol(\hatOobs)} \int_{\hatGobs} \frac{1}{2} x_i^2 \hat{n}_i \, d\hat s(\hat x)\\
&= \frac{1}{2\vol(\hatOobs)} \int_{\Gobs} (x_i + w_i)^2 \frac{1}{\Vert (D\tau)^{-\top} n \Vert_2} \left[(D\tau)^{-\top} n\right]_i \Vert J_\tau (D\tau)^{-\top} n \Vert_2 \, ds(x)\\
&= \frac{1}{2 \vol(\hatOobs)} \int_{\Gobs} (x_i + w_i)^2 \left[(D\tau)^{-\top} n\right]_i \vert J_\tau\vert \, ds(x).
\end{aligned}
\end{equation}
for the $i$-th component of the barycenter.
Hence, with the assumptions that the barycenter of the initial shape fulfills $\bc(\Oobs)_i = 0$ and
$J_\tau \geq \eta_1 > 0 $ we obtain the constant volume condition
\begin{equation}
\label{eq::fixed_volume}
\int_{\Gobs} (x+w)^\top (D\tau)^{-\top} n J_\tau  - x^\top n \,ds(x) = 0
\end{equation}
and the barycenter condition reduces to
\begin{equation}
\int_{\Gobs} (x_i + w_i)^2 \left[(D\tau)^{-\top} n\right]_i J_\tau \, ds(x) = 0.
\end{equation}
In the sequel we shortly write $ds$ instead of $ds(x)$.

\subsection{On the different strategies for $S$}

In \cref{sec::generalformulation} we discuss one particular choice of the operator $S$. We extend this by two further options.
In general, the operator $S$ involves solving an equation of Laplace-Beltrami type and an elliptic extension equation.
Thereby, the scalar-valued control variable $\control$ is mapped from the shape boundary $\Gobs$ to a vector-valued displacement field $w$ in $\Omega$.
The major difference in the considered strategies is when the variable becomes vector-valued.
We thus consider a mapping given by
\begin{equation}\label{eq::control_to_state_map}
\control\; \overset{i)}{\mapsto} \;b \;\overset{ii)}{\mapsto}\; z \overset{iii)}{\mapsto}\; w
\end{equation}
where i) is realized via the Laplace-Beltrami solution operator 
on $\Gobs$ and ii) via a solution operator for an elliptic equation in $\Omega$.
Depending on when the variables becomes vector-valued the auxiliary $z$ and step iii) is optional.
We start by recalling the strategy introduced and investigated in \cref{sec::s1} and 
then numerically test two further strategies.

Note that of the following choices for the operator $S$ only strategy S1 is entirely covered by the lemmas in \cref{sec::generalformulation}. For assumption \ref{assumption2} \cref{lem::fulfillsassumptions} can be applied in all three cases. In particular, our analysis in \cref{sec::generalformulation} can be used to show assumption \ref{assumption1} for strategy S1. It remains to verify assumption \ref{assumption1} for S2 and S3.
Nevertheless, we propose and numerically investigate S2 and S3 due to their computational attractiveness. 

\subsubsection*{First strategy (S1)}\label{S1}
This strategy only allows for displacements of $\Gobs$ along normal directions (cf.\ \cref{sec::s1}). We choose 
\begin{equation}
S(\control) \coloneqq S_\Omega(S_{\Gobs}(\control)) \Next,
\end{equation}
where $\Next$ denotes an extension of the outer unit normal vector field to $\Omega$. 
The corresponding weak formulation for the operators $S_{\Gobs}$ and $S_\Omega$ (step i) and ii), respectively) is given by
\begin{alignat}{2}
\label{eq::scalar_extension}
\int_\Omega \nabla z\cdot \nabla \psi_z \, dx =& \int_\Gobs b\psi_{z} \, ds &\quad \forall \psi_{z}
\\
\label{eq::laplace_beltrami_scalar}
\int_\Gobs b\psi_{b} + \nabla_{\Gobs} b \cdot \nabla_{\Gobs} \psi_{b} \, ds =& \int_\Gobs \control \psi_{b} \, ds &\quad \forall \psi_{b}.
\end{alignat}
Since our intention is to formulate everything suitable for weak form languages of the major FEM toolboxes, we
realize step iii) in the form
\begin{equation}
\label{eq::vect_L2_proj}
\int_\Omega w \cdot \psi_n \, dx = \int_\Omega z \Next \cdot \psi_n \, dx \quad \forall \psi_n.
\end{equation}
\subsubsection*{Second strategy (S2)}\label{S2}
As a second strategy we consider 
\begin{equation}
S (\control) \coloneqq S_\Omega^d(S_{\Gobs}(\control)n),
\end{equation}
where $n$ denotes the outer unit normal vector field on $\Gobs$.
Thus, the elliptic extension equation in step ii) (corresponding to the operator $S_\Omega^d$) is defined to be vector-valued, which in terms allows to omit step iii).
This reads in weak formulation as
\begin{equation}
\label{eq::vector_extension}
\int_\Omega (Dw + Dw^\top) : D \psi_w \, dx = \int_\Gobs bn \cdot \psi_{w} \, ds \quad \forall \psi_{w}
\end{equation}
and replaces \eqref{eq::scalar_extension}.
Note that we use the symmetrized derivative $(Dw + Dw^\top)$ in \cref{eq::vector_extension}, which corresponds to solving the Lam\'e system with 
Lam\'e parameters $\mu = 1$ and $\lambda = 0$ and is found out to lead to better 
mesh qualities after deformation compared to using $Dw$ instead. 
With our approach it is not 
required to tune these parameters contrary to previous approaches, see e.g.
\cite{schulz2016computational, DokkenFunke20}.
This is later substantiated with numerical results in \cref{fig::2d_nose_grid}.
Furthermore, equation \eqref{eq::vect_L2_proj} is dropped from the system.

\subsubsection*{Third strategy (S3)}\label{S3}
In a third possible strategy the scalar-valued control $\control$ is immediately mapped to a vector-valued $b$ in step i) by the Laplace-Beltrami solution operator.
We obtain the following representation

\begin{equation}
S(d)\coloneqq S_\Omega^d(S_{\Gobs}^d(\control n)),
\end{equation}
where again $n$ is the unit outer normal field at $\Gobs$.
Note that the scalar-valued control $\control$ enters as a scaling of $n$ and then a vector-valued Laplace-Beltrami type equation is considered.
We denote the corresponding vector-valued solution operator by $S_{\Gobs}^d$ which is given in the following weak formulation
\begin{equation}
\label{eq::vector_laplace_beltrami_scalar}
\int_\Gobs b \cdot \psi_{b} + D_{\Gobs} b : D_{\Gobs} \psi_{b} \, ds = \int_\Gobs \control n \cdot \psi_{b} \, ds \quad \forall \psi_{b}.
\end{equation}
The operator $S_\Omega^d$ is the same as in S2 and given in weak form by \cref{eq::vector_extension}.
\subsection{Optimality system}
\label{sec::optsys}
We present the optimality system for strategy S3. Strategies S1 and S2 
can be handled analogously. 
Using that the weak formulation of 
the transformed Stokes equations is given by
\begin{multline}
\int_{\Omega} \left( D v (D\tau)^{-1} \right) :  \left( D \mtestv (D\tau)^{-1} \right)  J_\tau \, dx - \int_{\Omega}  p \Tr \left( D \mtestv (D\tau)^{-1} \right) J_\tau\, dx\\
+ \int_\Omega \mtestp \Tr (D v(D\tau)^{-1}) J_\tau \,dx = 0 \quad \forall \mtestv, \mtestp,
\end{multline}
the Lagrangian for the energy dissipation minimization problem of a Stokes flow around an obstacle with fixed volume and barycenter is given by
\begin{multline}
\label{eq::Lagrangian}
\L ( w, v, p, b, \mtestw,\mtestv,\mtestp, \mtestb, \control, \lambdav, \lambdabc) =\\
\frac{1}{2} \int_{\Omega} \left( D v (D\tau)^{-1} \right) :  \left( D v (D\tau)^{-1} \right)  J_\tau\, dx
+\frac{\alpha}{2}\int_\Gobs \control^2 \, ds
+ \frac{\gamma_1}{2} \int_{\Omega} (( \eta_1 - J_\tau )_+ )^2 \, dx \\
- \int_{\Omega} \left( D v (D\tau)^{-1} \right) :  \left( D \mtestv (D\tau)^{-1} \right)  J_\tau \, dx + \int_{\Omega}  p \Tr \left( D \mtestv (D\tau)^{-1} \right) J_\tau\, dx\\
 - \int_\Omega \mtestp \Tr (D v(D\tau)^{-1}) J_\tau \,dx - \int_\Omega (D w + D w^\top): D \mtestw \, dx + \int_\Gobs b\cdot \mtestw \, ds\\
- \int_\Gobs b\cdot \mtestb + D_{\Gobs} b : D_{\Gobs} \mtestb \, ds + \int_\Gobs \control n \cdot \mtestb \, ds \\
+ \sum_{i=1}^d\lambdabc_i \int_{\Gobs} (x_i + w_i)^2 \, \left((D\tau)^{-\top} n\right)_i J_\tau \, ds + \frac{\lambdav}{d} \int_{\Gobs} (x+w)^\top (D\tau)^{-\top} n J_\tau - x\cdot n \, ds,
\end{multline}
where
$\psi_{(\cdot)}$ denotes the adjoint states. 

For the sake of simplicity we write in the sequel $\L$ for $\L ( w, v, p, \mtestw,\mtestv,\mtestp,\control, \lambdav, \lambdabc)$.
Using $((D\tau)^{-1})_w \deltaw = - (D \tau)^{-1} D \deltaw (D \tau)^{-1}$ and 
$(J_\tau)_w \deltaw = \Tr((D\tau)^{-1} D \deltaw) J_\tau$, 
the first order necessary optimality conditions are given by
\begingroup
\allowdisplaybreaks
\begin{align}
\L_{w}  \deltaw = &
- \int_\Omega ( Dv (D \tau)^{-1}): (D v (D \tau)^{-1} D \deltaw (D \tau)^{-1})  J_\tau \, dx \notag\\ &
+ \frac12 \int_\Omega (Dv (D \tau)^{-1}): (Dv (D \tau)^{-1})\Tr((D\tau)^{-1} D \deltaw) J_\tau \, dx \notag\\ &
- \gamma_1 \int_\Omega (\eta_1 - J_\tau)_+ \Tr((D\tau)^{-1} D \deltaw) J_\tau \, dx\notag\\&
+ \int_\Omega (Dv (D \tau)^{-1} D \deltaw (D \tau)^{-1}) : ( D \mtestv (D \tau)^{-1}) J_\tau \, dx \notag\\& 
+ \int_\Omega (Dv (D \tau)^{-1}): ( D \mtestv (D \tau)^{-1} D \deltaw (D \tau)^{-1}) J_\tau \, dx \notag\\& 
- \int_\Omega (Dv (D \tau)^{-1}) : ( D \mtestv (D \tau)^{-1})\Tr((D\tau)^{-1} D \deltaw) J_\tau \, dx \notag\\&
- \int_\Omega p \Tr( D \mtestv (D \tau)^{-1} D \deltaw (D \tau)^{-1}) J_\tau \, dx \notag\\&
+ \int_\Omega p \Tr( D \mtestv (D \tau)^{-1}) \Tr((D\tau)^{-1} D \deltaw) J_\tau \, dx \notag\\&
+ \int_\Omega \mtestp \Tr(D v (D \tau)^{-1} D \deltaw (D \tau)^{-1}) J_\tau \, dx \label{eq::opt_sys_w}\\&
- \int_\Omega \mtestp \Tr( D v (D \tau)^{-1}) \Tr((D\tau)^{-1} D \deltaw) J_\tau \, dx \notag\\&
- \int_\Omega ( D \deltaw + D \deltaw^\top) : D \mtestw\, dx \notag\\&
+ \sum_{i=1}^d \mu_i \int_\Gobs 2 (x_i + w_i) (\deltaw)_i ((D \tau)^{-\top} n)_i J_\tau \, dx \notag\\&
- \sum_{i=1}^d \mu_i \int_\Gobs (x_i + w_i)^2 ((D \tau)^{-\top} (D \deltaw)^\top (D \tau)^{-\top} n)_i J_\tau \, dx \notag\\&
+ \sum_{i=1}^d \mu_i \int_\Gobs (x_i + w_i)^2 ((D \tau)^{-\top} n)_i \Tr((D\tau)^{-1} D \deltaw) J_\tau\, dx \notag\\&
+ \frac\lambda{d} \int_\Gobs (x + \deltaw)^\top (D \tau)^{-\top} n J_\tau \, ds \notag\\&
- \frac\lambda{d} \int_\Gobs (x + w)^\top (D \tau)^{-\top} (D \deltaw)^\top (D \tau)^{-\top} n J_\tau \, ds \notag\\&
+ \frac\lambda{d} \int_\Gobs (x + w)^\top (D \tau)^{-\top} n \Tr((D\tau)^{-1} D \deltaw) J_\tau\, ds = 0, \notag
\end{align}
\endgroup
\begin{align}
\begin{split}
\L_{v}  \deltav = & \int_{\Omega} \left( D \deltav (D\tau)^{-1} \right) :  \left( D v (D\tau)^{-1} \right)  J_\tau\, dx\\
&- \int_{\Omega}  \left( D \deltav (D\tau)^{-1} \right) :  \left( D \mtestv (D\tau)^{-1} \right)  J_\tau \,dx
- \int_{\Omega} \mtestp \Tr (D \deltav(D\tau)^{-1}) J_\tau \,dx = 0,
\end{split}
\label{eq::opt_sys_v}
\end{align}
\begin{equation}
\label{eq::opt_sys_p}
\L_{p}  \deltap = \int_\Omega \deltap \Tr \left( D \mtestv (D\tau)^{-1} \right) J_\tau\, dx = 0,
\end{equation}
\begin{align}
\label{eq::opt_sys_adj_v}
\L_{\mtestv} \deltatestv = &
-\int_{\Omega}  \left( D v (D\tau)^{-1} \right) :  \left( D \deltatestv (D\tau)^{-1} \right)  J_\tau \,dx \\ 
\notag &+ \int_{\Omega} p \Tr (D \deltatestv(D\tau)^{-1}) J_\tau \,dx = 0,
\end{align}
\begin{equation}
\label{eq::opt_sys_adj_p}
\L_{\mtestp}  \deltatestp = - \int_\Omega \deltatestp \Tr \left( D v (D\tau)^{-1} \right) J_\tau\, dx = 0,
\end{equation}
\begin{equation}
\L_{\mtestw}  \deltatestw =
- \int_\Omega (D w + D w^\top): D \deltatestw \, dx + \int_\Gobs b\cdot \deltatestw \, ds = 0,
\end{equation}
\begin{equation}
\L_{b}  \deltab =
- \int_\Gobs \deltab\cdot \mtestb + D_{\Gobs} \deltab : D_{\Gobs} \mtestb \, ds + \int_\Gobs \deltab \cdot \mtestw \, ds = 0,
\end{equation}
\begin{equation}
\label{eq::deltatestb}
\L_{\mtestb}  \deltatestb =
- \int_\Gobs b\cdot \deltatestb + D_{\Gobs} b : D_{\Gobs} \deltatestb \, ds + \int_\Gobs \control n \cdot \deltatestb \, ds = 0,
\end{equation}
\begin{equation}
\label{eq::control} 
\L_{\control}  \deltacontrol =
\alpha \int_\Gobs \control \deltacontrol \, ds + \int_\Gobs \deltacontrol n \cdot \mtestb \, ds 
= 0,
\end{equation}
\begin{equation}
\L_{\lambdav}  \deltalambdav = \frac{\deltalambdav}{d} \int_{\hatGobs} (x+w)^\top (D\tau)^{-\top} n J_\tau - x\cdot n \, ds = 0,
\end{equation}
\begin{equation}
\label{eq::opt_sys_mu}
\L_{\lambdabc}  \deltalambdabc =
\sum_{i=1}^d (\deltalambdabc)_i \int_{\Gobs} (x_i + w_i)^2 \, ((D\tau)^{-\top} n)_i  J_\tau \, ds = 0,
\end{equation}
for all $(\deltaw, \deltav, \deltap, \deltatestw, \deltatestv, \deltatestp, \deltacontrol, \deltalambdav, \deltalambdabc)$ in appropriate function spaces.
We thus obtain a system of nonlinear, coupled PDEs in a suitable form for standard finite element toolboxes. 
\subsection{On the semismoothness of the optimality system}
\label{sec::semismoothness}
We solve the system \eqref{eq::opt_sys_v}-\eqref{eq::opt_sys_mu} with a semismooth Newton method. To justify this, we show semismoothness of
the system and therefore take a closer look at the term in \eqref{eq::opt_sys_v} that appears 
by differentiating
\begin{align*}
\frac12 \int_\Omega ((\eta_1 - J_\tau)_+)^2 dx = \int_\Omega (f_2 \circ \iota \circ f_1 (w))(x) dx = F \circ \iota \circ f_1
\end{align*}
with 
\begin{align*}
 f_1: ~&H^s (\Omega)^d \to H^{s-1}(\Omega), \quad w \mapsto \eta_1 - J_\tau, \\
 \iota: ~&H^{s-1} (\Omega) \to L^r(\Omega), \quad v \mapsto v, \\
 f_2:~ &L^r(\Omega) \to L^1(\Omega), \quad q \mapsto \frac12 (q)_+^2, \\
 F: ~&L^r(\Omega) \to \mathbb R, \quad q \mapsto \int_\Omega \frac12(q)_+^2 dx
\end{align*}
and $2 \leq r \leq \infty$.
Since $H^{s-1}(\Omega)$ is 
a Banach algebra for $s> 1 + \frac{d}2$, $f_1: ~H^s(\Omega)^d \to H^{s-1}(\Omega)$ is $\mathcal C^\infty$. Since $s -1 -\frac{d}2 > 0$, the embedding $\iota$ is linear and continuous. 
The Nemytskii operator $f_2: L^r(\Omega) \to L^1(\Omega)$ is Fr\'echet differentiable for $r \geq 2$, see e.g. \cite[Sec. 4.3.3]{TrltzschBook}, and thus 
$F: q \mapsto \int_\Omega \frac12 (q)_+^2 dx$ 
is Fr\'echet 
differentiable as a mapping $L^r(\Omega) \to \mathbb R$ for $r \geq 2 $ with 
derivative 
$F^\prime (q): L^r(\Omega) \to \mathbb R, ~ h \mapsto \int_\Omega (q)_+ h dx$. Let 
$2 \leq r < \infty$. Then $F^\prime \in L^r(\Omega)^*$ as an element of the dual space
of $L^r(\Omega)$ can be identified with $F^\prime(q) = (q)_+ \in L^{r^\prime}(\Omega)$ where $ r^\prime = \frac{r}{r-1}$. Now, by \cite[Thm. 3.49]{UlbrichBook}, $q \mapsto (q)_+$ is locally Lipschitz and semismooth as a mapping
$L^r(\Omega) \to L^{r^\prime}(\Omega)$ for $r > 2$, which implies semismoothness of 
$w \mapsto F^{\prime} \circ \iota \circ f_1$ as a mapping $H^s(\Omega)^d \to L^{r^\prime}(\Omega)$ by \cite[Prop. 3.8]{UlbrichBook}. Hence, since 
$H^{s-1}(\Omega) \hookrightarrow L^\infty(\Omega)$ for $s> 1 + \frac{d}2$, the mapping
\begin{equation}
G: H^s(\Omega)^d \to (H^{s}(\Omega)^d)^*, \quad G(w)(\deltaw) := \int_\Omega (\eta_1 - J_\tau)_+ \Tr((D\tau)^{-1} D \deltaw) J_\tau \, dx
\label{eq::G-semismooth}
\end{equation}
is semismooth.
\begin{algorithm}
  \caption{Optimization algorithm}
  \label{alg::opt_alg}
  \begin{algorithmic}[1]
    \Require $0 < \atarget \leq \ainit$, $0 < \adec < 1$, $0 \leq \gamma_1$,
    $0 < \eta_1$,
    $n_{\text{ssn}}$,
    $\epsilon_\text{ssn}$
    \State Initialize all variables $( w, v, p, b, \mtestw,\mtestv,\mtestp, \mtestb, \control, \lambdav, \lambdabc)_0$ with zero
    \State $k \gets 0$, $\alpha_k \gets \ainit$
    \While {$\alpha_k \geq \atarget$}
    \Repeat
    \State \begin{varwidth}[t]{\linewidth}
      Solve \eqref{eq::opt_sys_v}-\eqref{eq::opt_sys_mu} for $(w, v, p, b, \mtestw, \mtestv, \mtestp, \mtestb, \control, \lambdav, \lambdabc)_{k+1}$ with \\ semismooth Newton method, $( w, v, p, b, \mtestw,\mtestv,\mtestp, \mtestb, \control, \lambdav, \lambdabc)_{k}$ \\ 
       as initial guess and regularization parameter $\alpha_k$
    \end{varwidth}\label{alg::semismooth}
    \If {Newton's method not converge to $\epsilon_\text{ssn}$ within $n_{\text{ssn}}$ iterations}
    \State $\alpha_{k} \gets \frac{1}{2} ( \frac{\alpha_k}{\adec} - \alpha_k)$
    \EndIf
    \Until Newton's method converged
    
    \State $\alpha_{k+1} \gets \adec \alpha_k$
    \State $k\gets k+1$
    \EndWhile
  \end{algorithmic}
\end{algorithm}

\subsection{Numerical Results}
\label{sec::numerics}

In this section we demonstrate the three proposed strategies S1-S3 in a two-dimensional (2d) and a three-dimensional (3d) case.
In both cases we consider a Stokes fluid in a flow tunnel with an obstacle in the center. Starting from a circular shape (in 2d) and a sphere (in 3d) the task is to optimize the shape such that the energy dissipation measured over the domain is minimized.
This is a classical test case, which is investigated in detail for instance in \cite{mohammadi2010applied}.

The experimental settings in 2d are given by a rectangular domain $\Omega = [-10,10]\times[-3,3]$ where the initial obstacle is a circle with radius $0.5$ and barycenter at $ (0,0)^\top$.
We consider a flow along the $x_1$-axis which is modeled by the inflow velocity profile
\begin{equation}
v^\infty_{x_1} = \cos(\frac{2 \Vert x \Vert_2 \pi}{\delta})
\end{equation}
where $\delta$ specifies the diameter of the inflow boundary in both 2d and 3d.
This is consistent with the zero-velocity boundary conditions at the walls of the flow tunnel.

The discretization of the domains is performed with the Delaunay method within the toolbox GMSH \cite{GMSH}.
In 2d we choose three different hierarchical grids with \num{1601}, \num{6404} and \num{25616} triangles.
After each refinement the grid at $\Gobs$ is adapted to interpolate the circular obstacle and consists of \num{141}, \num{282} and \num{564} line segments.

The 3d experiment is conducted in a cylindrical domain
\begin{equation*}
\Omega = \lbrace x \in \R^3: -10 \leq x_1 \leq 10, \sqrt{x_2^2 + x_3^2} \leq 3 \rbrace
\end{equation*}
where the initial obstacle is a sphere of radius $0.5$ with barycenter $(0,0,0)^\top$.
In this situation $\Omega$ is discretized with \num{6994} surface triangles forming $\Gobs$ and \num{118438} tetrahedrons in the volume.

For all numerical computations in this section we use the PDE toolbox GETFEM++ \cite{GETFEM}.
We utilize the parallelized version of this library and provide the nonlinear optimality system \crefrange{eq::opt_sys_w}{eq::opt_sys_mu} in the builtin language for weak formulations as it is.
In order to solve the nonlinear system second derivatives are computed symbolically by the library.
While all terms but one in \crefrange{eq::opt_sys_w}{eq::opt_sys_mu} are classically differentiable with respect to $w$, the integral in \cref{eq::opt_sys_w}, which involves the non-differentiable positive-part function $(\eta_1 - J_\tau)_+$, leads to a generalized derivative. Following the discussion in \cref{sec::semismoothness} of the semismoothness of the operator $G$ in \cref{eq::G-semismooth} we obtain for the assembly of the linearization matrix
\begin{multline}
\gamma_1 \int_\Omega \chi_{(\eta_1 > J_\tau)} \Tr((D\tau)^{-1} D \deltawsec)  \Tr((D\tau)^{-1} D \deltaw) J_\tau^2\\
+ (\eta_1 - J_\tau)_+ \Tr((D\tau)^{-1} D\deltawsec (D\tau)^{-1} D \deltaw) J_\tau\\
- (\eta_1 - J_\tau)_+ \Tr((D\tau)^{-1} D \deltaw) \Tr((D\tau)^{-1} D\deltawsec) J_\tau\, dx
\label{eq::second-derivative-w}
\end{multline}
for all $\deltaw, \deltawsec$. Corresponding to \cite[(4.1)]{UlbrichSemismooth} we can identify
$$-\chi_{(\eta_1 > J_\tau)} \Tr((D\tau)^{-1} D \deltawsec) J_\tau$$
in \cref{eq::second-derivative-w} with an element of the generalized differential of $(\eta_1 - J_\tau)_+$ evaluated in a direction $\deltawsec$.

For the discretization of the linearization matrix and the right hand side in Newton's method we choose piece-wise linear basis functions for all variables except for the velocity $v$ and its adjoint $\mtestv$. Here we choose piece-wise quadratic functions.
For simplicity, in each iteration of Newton's method for the system \crefrange{eq::opt_sys_w}{eq::opt_sys_mu} the parallel direct LU solver MUMPS \cite{MUMPS} is applied.

\begin{figure}
\begin{center}
\includegraphics[width=0.45\textwidth]{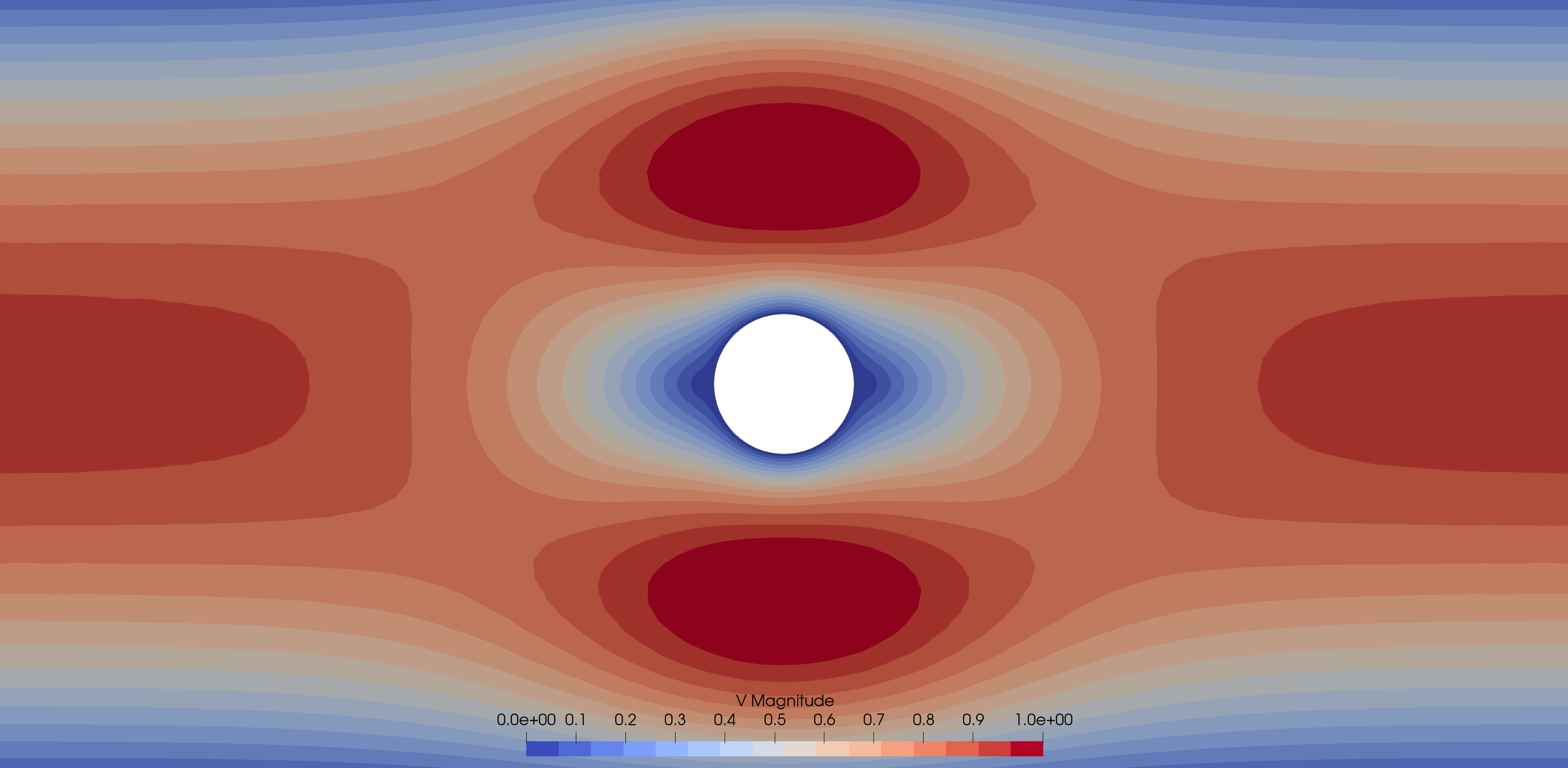}\hspace{0.02\textwidth}
\includegraphics[width=0.45\textwidth]{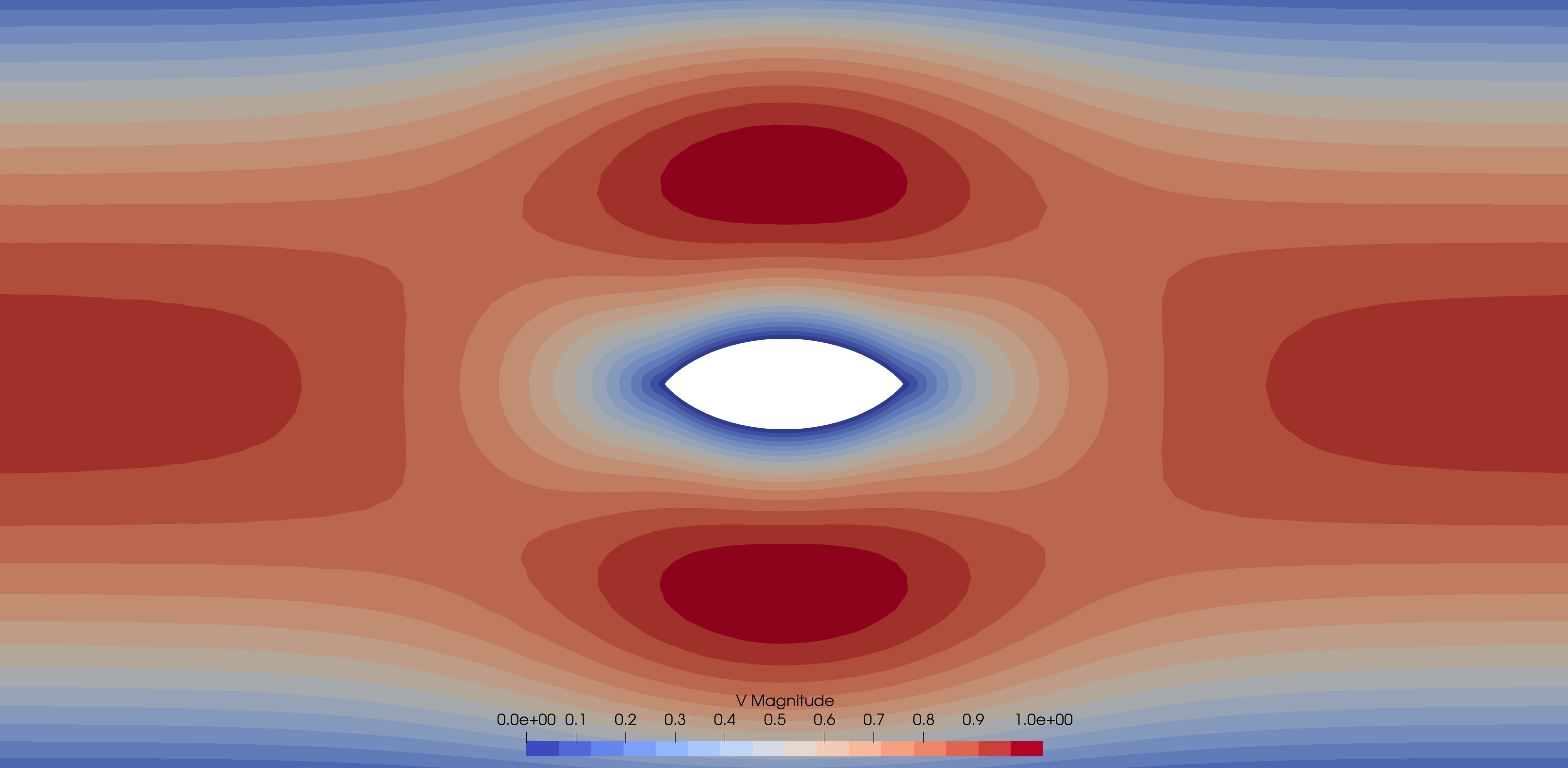}
\end{center}
\caption{Holdall domain $\holdall$ and Stokes flow in $\Omega = \holdall \setminus \Oobs$ on the left and the optimal, deformed configuration $\hatOobs = \tau(\Oobs)$ on the right. Color denotes $\Vert v\Vert$ and $\Vert \hat{v} \Vert$, respectively.}
\label{fig::2d_experiment}
\end{figure}
\begin{figure}
\begin{center}
\includegraphics[width=0.95\textwidth]{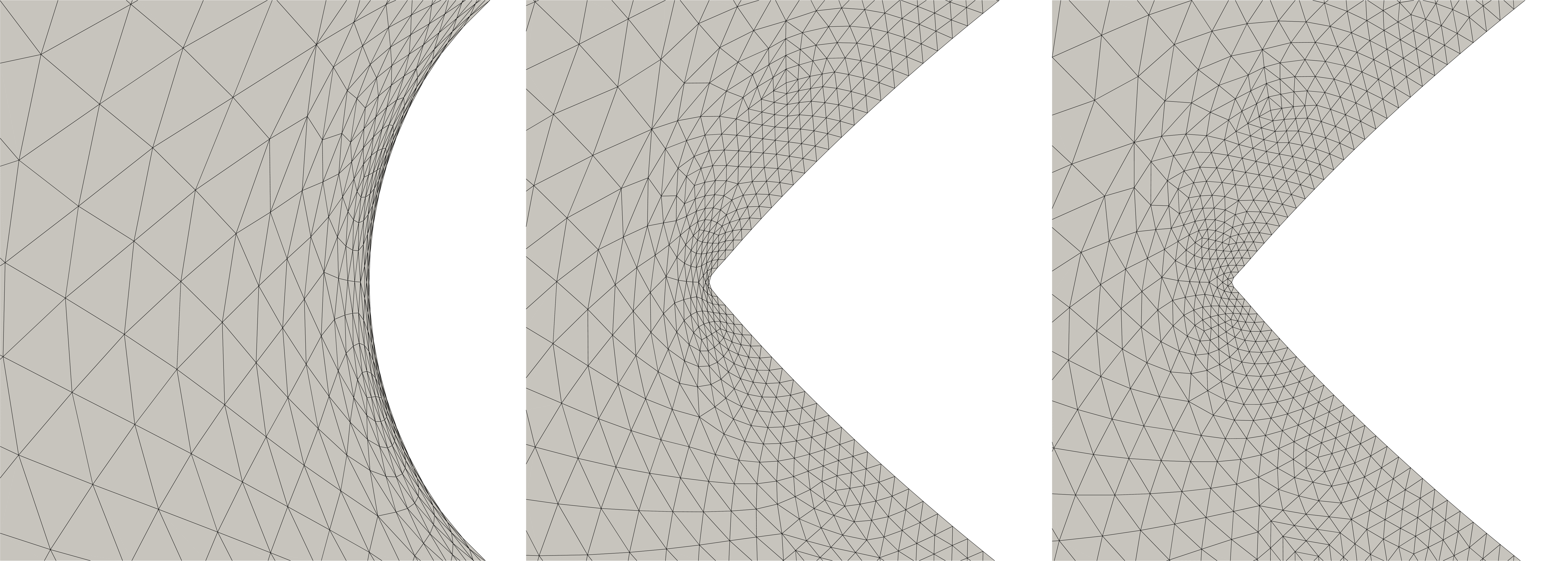}
\end{center}
\caption{Optimal solution for regularization parameter $\atarget = 10^{-10}$ following strategy S1,S2 and S3 (from left to right). The images show a $0.28 \times 0.28$ section centered at the point $(-0.8, 0.0)^\top$.}
\label{fig::2d_nose_grid}
\end{figure}
\begin{figure}
	\begin{center}
		\includegraphics[width=1.0\textwidth]{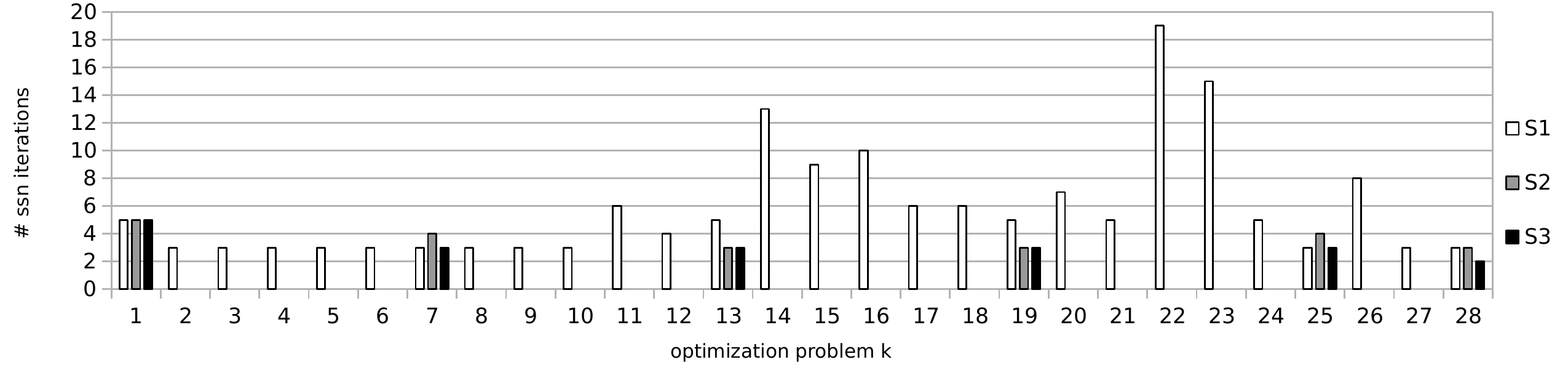}
	\end{center}
	\caption{Semismooth Newton iteration counts with a tolerance of relative residual $\epsilon_\text{ssn}=\num{1e-9}$ for each subsequent optimization problem $k$ with $\alpha=\num{1e-2}\cdot\frac{1}{2}^{k-1}$, $\atarget=\num{1e-10}$. For S2 and S3 $\adec = \frac{1}{64}$ is chosen, thus intermediate problems are left out.}
	\label{fig::ssn-iterations}
\end{figure}
\cref{fig::2d_experiment} depicts the 2d situation where color denotes the norm of the velocity field.
The velocity profile in the 3d experiment is similar to the one shown in \cref{fig::2d_experiment} since we choose the domain $\Omega$ in 3d to be the rotation body of the 2d domain.

In all experiments in this section $\epsilon_\text{ssn} = \num{1e-9}$ is chosen as tolerance of the relative residual norm in the semismooth Newton method in \cref{alg::opt_alg}.
Further, if the criterion is not fulfilled after $n_{\text{ssn}} = 40$ steps, $\alpha$ is increased again. 

In \cref{fig::2d_nose_grid} we compare the optimal solution for a regularization factor of $\atarget = 10^{-10}$ for the strategies S1, S2 and S3 on the finest grid with \num{25616} triangles and \num{564} surface elements.
Here the effect of the tangential movements of nodes can be seen.
While in strategy S1 in the leftmost figure the optimal shape stays round at the tip, strategy S2 and S3 approximate the kink.
The same holds true for the back of the shape, which is not shown here.
Since the resulting deformation field $w$ restricted to $\hatGobs$ in S1 points in normal direction,
the condition $J_\tau = \det(I + Dw) \geq \eta_1 > 0$ prevents the appearance of a kink.
Numerical tests show that the choice of $\Next$ plays a decisive role.
Since the reference shape $\Oobs$ is either a circle in 2d or a sphere in 3d with barycenter zero one can choose $\Next(x) = \frac{x}{\Vert x \Vert_2}$ as an extension to the normal vector field on $\Gobs$. The numerical results for S1 presented here are obtained for the choice $\Next(x) = (\frac{1}{2}+\Vert x \Vert_2)^2 x$.
Numerical experiments have shown that with the second choice of $\Next$ we come closer to the optimal shapes resulting from S2 and S3 than with the first variant.

\begin{figure}
\begin{center}
\includegraphics[height=0.15\textheight]{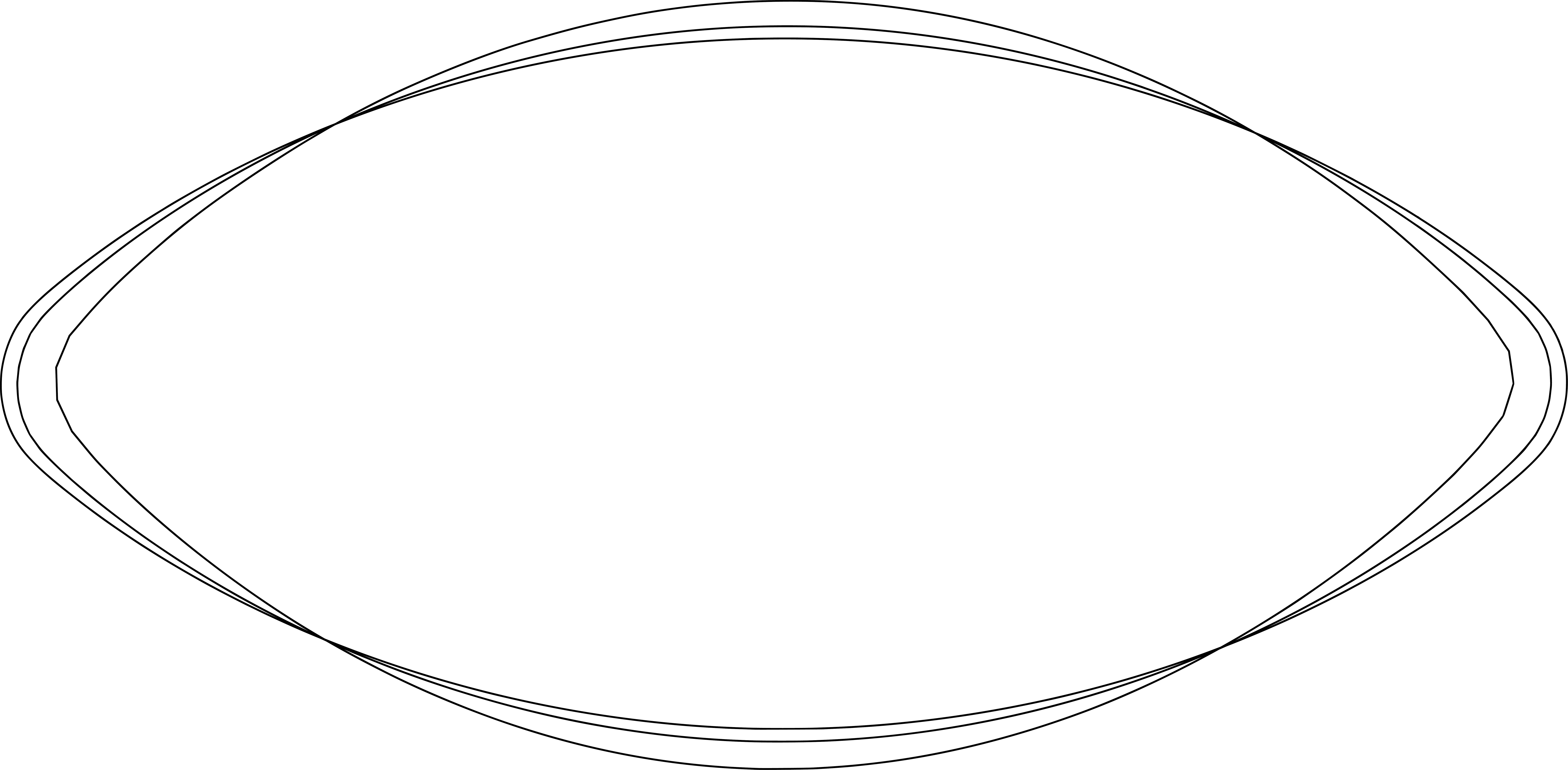}\hspace{1.0cm}
\includegraphics[height=0.15\textheight]{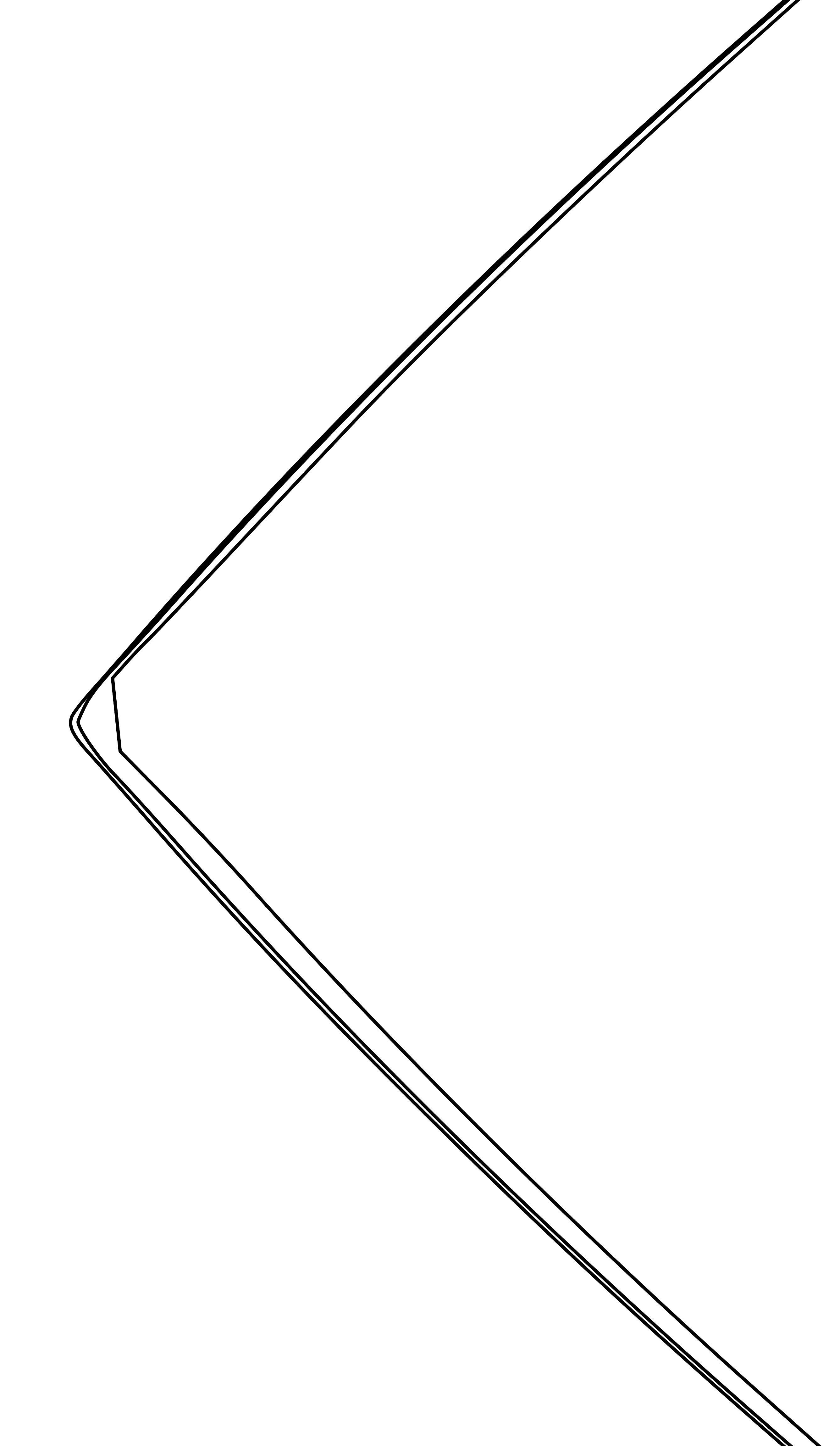}
\end{center}
\caption{Optimal solution for regularization parameter $\atarget = 10^{-10}$ under grid refinements $j=1,2,3$, i.e.\ $1601 \cdot 4^{j-1}$ triangles, $141 \cdot 2^{j-1}$ surface lines. Strategy S1 on the left hand side and S3 with a zoom on the nose of the shape.}
\label{fig::2d_h_ref}
\end{figure}

In \cref{fig::ssn-iterations} the number of semismooth Newton iterations is depicted for each of the optimization problems.
According to \cref{alg::opt_alg} we utilize the optimal control of one problem as initialization for the next one with smaller regularization parameter $\alpha$.
Computations are performed on the finest 2d grid considered in this section, i.e.\ $j=3$.
For all three strategies S1,S2 and S3 we choose $\ainit = \num{1e-2}$  and $\atarget=\num{1e-10}$.
While for S1 $\adec=\frac{1}{2}$ is required to guarantee convergence of the semismooth Newton method within $n_\text{ssn}=40$ we proceed with $\adec=\frac{1}{64}$ for S2 and S3.
We observe that the number of required iterations significantly increases beginning in the $14$th optimization problem for strategy S1.
This can be explained by the positive-part in the objective of \cref{optprobstokespen} becoming active.

In the next experiment we consider strategies S1 and S3 under mesh refinements.
\cref{fig::2d_h_ref} shows the corresponding results for three hierarchically refined grids resulting in $1601 \cdot 4^{j-1}$ triangles and $141 \cdot 2^{j-1}$ surface lines for $j=1,2,3$.
The regularization parameter is again chosen as $\atarget = 10^{-10}$.
The right hand figure shows a zoom-in to the $0.28\times 0.28$ square around the tip in order to make the shapes distinguishable.
On the left hand side, i.e.\ where there are only deformations in normal direction, we observe a slow grid-convergence towards the theoretical, optimal shape.
Strategy S3, in contrast, leads to comparable results even on relatively coarse grids.

\begin{figure}
\begin{center}
\includegraphics[width=0.5\textwidth]{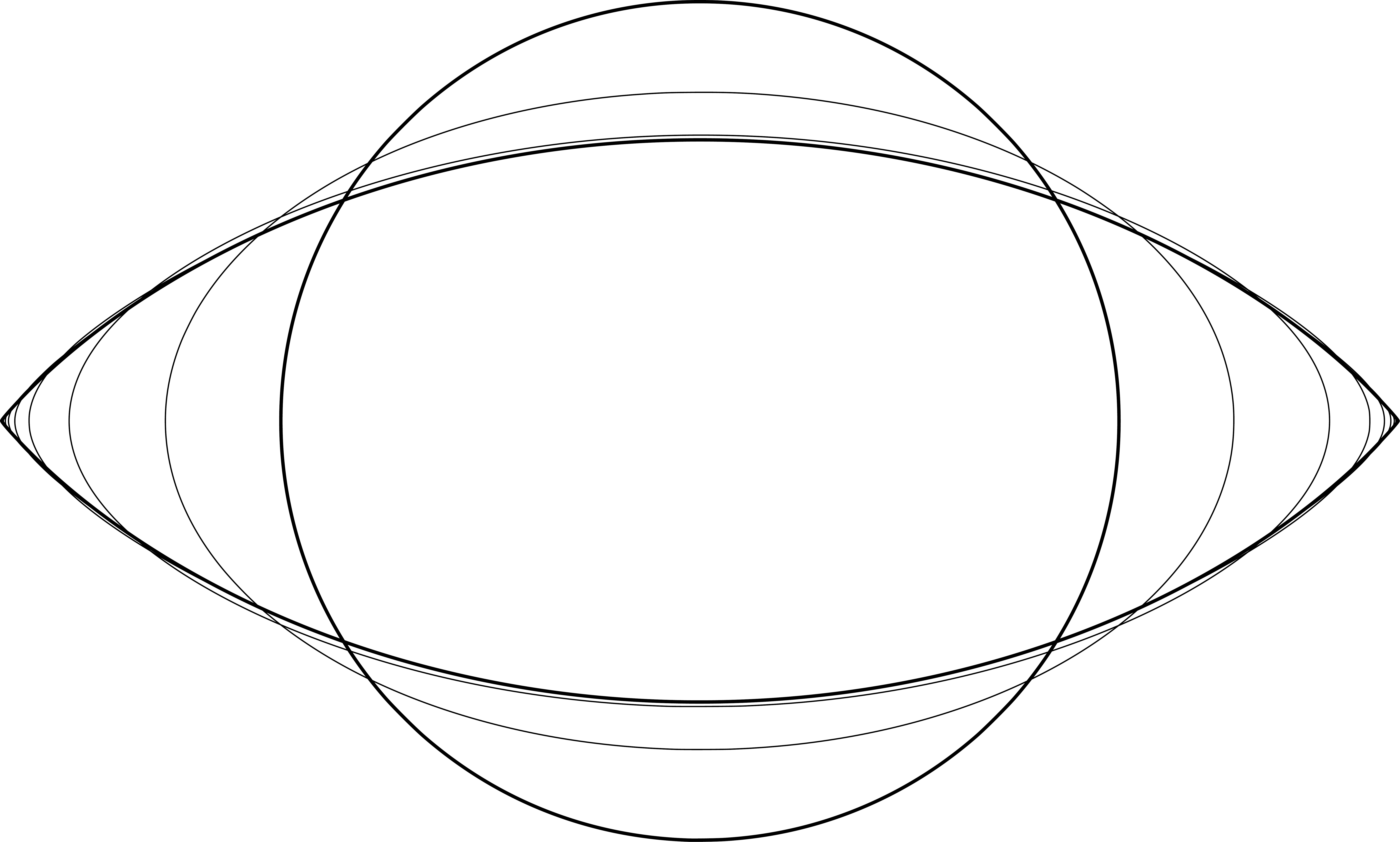}
\end{center}
\caption{Optimal solution with regularization parameter $\alpha = 10^{-k}$ for $k=0,\dots,10$ according to strategy S3.}
\label{fig::2d_alpha}
\end{figure}
\cref{fig::2d_alpha} visualizes the effect of the regularization parameter $\alpha$.
More precisely, a sequence of optimal shapes for different optimization problems depending on $\alpha$ are illustrated.
The figure shows a transition for $\alpha = 10^{-k}$ for $k=0,\dots,10$ according to strategy S3 on the finest grid, i.e. it presents 
the intermediate, optimal solutions one obtains after each iterations of \cref{alg::opt_alg}.
It should be mentioned that this fine resolution in $\alpha$ is chosen for demonstration purposes only.
For the specific example we are able to choose an initial and decrement factor for $\alpha$ such that $\atarget = 10^{-10}$ is reached in two iterations of \cref{alg::opt_alg}.  
Since we are only interested in the optimal shape with respect to $\atarget$ it is our intention to choose both $\ainit$ and $\adec$ in  \cref{alg::opt_alg} as small as possible.
This choice is made heuristically depending on whether the semismooth Newton method in line \cref{alg::semismooth} converges within a prescribed number of iterations.
If the inner iteration does not converge, we choose $\adec$ closer to one.
In all two dimensional computations we choose the parameter $\eta_1 = \num{8e-2}$, $\gamma_1 = \num{1e3}$ independently of the $\alpha$-strategy.

\begin{figure}
\begin{center}
\includegraphics[width=0.49\textwidth]{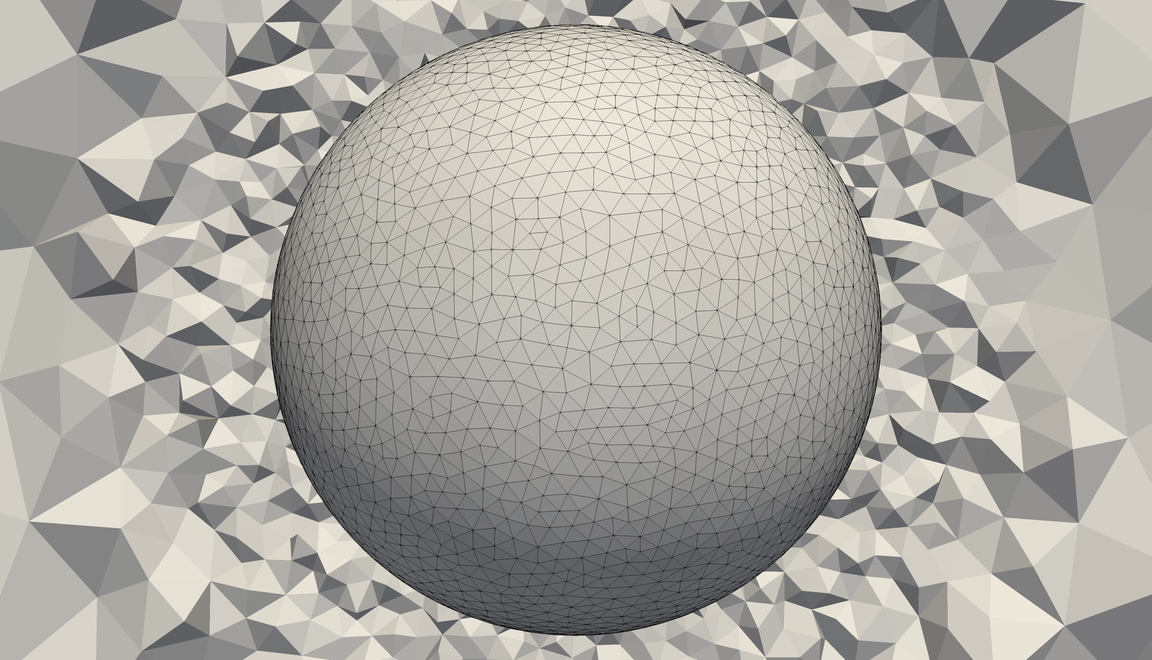}
\includegraphics[width=0.49\textwidth]{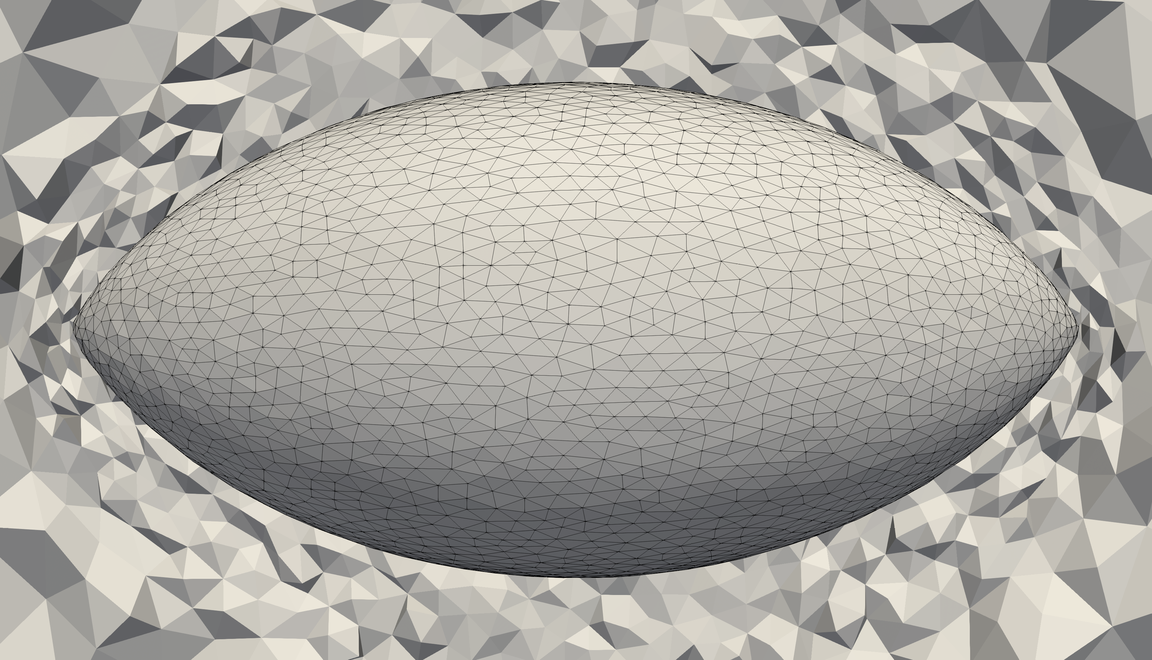}
\end{center}
\caption{Reference $\Oobs$ (left) and transformed shape $\hatOobs$ (right) according to optimal displacement $w$ in 3d Stokes flow  with a crinkled slice through the surrounding grid. The result is achieved with strategy S3 and $\atarget = 10^{-10}$.}
\label{fig::3d_results}
\end{figure}
\cref{fig::3d_results} visualizes \cref{alg::opt_alg} for 3d problems.
It visualize the reference shape $\Gobs$ as the surface triangulation together with a slice through the tetrahedral grid of the reference domain $\Omega$ in the left subfigure.
On the right hand side the effect of the optimal displacement field $w$ to the shape $\hatGobs$ and the volume $\hat{\Omega}$ is shown.
As mentioned above we are only interested in the optimal control $\control$ and the corresponding displacement field $w$ for the regularization parameter $\atarget$.
In the 2d examples this could be achieved with very few outer iterations of \cref{alg::opt_alg}, which means that one could start with a small $\ainit$ and proceed fast towards $\atarget$.
However, in the 3d case it turns out that a more careful strategy has to be considered in order to obtain convergence of Newton's method within $n_{\text{ssn}}$ steps.
The results shown in \cref{fig::3d_results} are obtained with $\ainit = \num{1e-1}$, $\atarget = \num{1e-6}$, $\adec = \num{0.5}$, $\eta_1 = \num{8e-2}$, $\gamma_1 = \num{1e3}$. 

\section{Conclusion and Outlook}
\label{sec::conclusion}
We present a formulation of shape optimization problems based on the method 
of mappings that is motivated from a continuous perspective. 
Using this approach replaces the problem of preventing mesh degeneration by the question of finding a suitable set of admissible transformations.
We propose a method such that the set of feasible transformations is a subset of the space of $\mathcal C^1$-diffeomorphisms. Numerical simulations substantiate the versatility of this approach.
Furthermore, it allows for refinement and relocation strategies
during the optimization process and can also be combined with adaptive mesh refinement strategies and globalized trust region methods. This, however, is left for future research.

\section*{Acknowledgments}
Johannes Haubner and Michael Ulbrich received support from the Deutsche
Forschungsgemeinschaft (DFG, German Research Foundation) 
as part of the International Research Training Group IGDK 1754
``Optimization and Numerical Analysis for Partial Differential Equations with Nonsmooth
Structures'' -- Project Number 188264188/GRK1754.
The work of Martin Siebenborn was partly supported by the DFG within the Research Training Group 2583 ``Modeling, Simulation and Optimization of Fluid Dynamic Applications''.

\bibliographystyle{siamplain}
\bibliography{literature}

\end{document}